\documentclass[smallextended,referee,envcountsect,]{svjour3}
\smartqed
\usepackage{graphicx}
\usepackage{mathptmx}
\usepackage{latexsym}
\usepackage{amsfonts}
\usepackage{mathrsfs}

\usepackage{amsmath}
\usepackage{amssymb}
\journalname{None}

\newcommand{\Pb}{\mathbb P}
\newcommand{\R}{\mathbb R}
\newcommand{\N}{\mathbb N}

\newcommand{\Z}{\mathbb Z}

\newcommand{\X}{\mathbb X}
\newcommand{\Y}{\mathbb Y}
\newcommand{\Uball}{{\mathbb B}}
\newcommand{\Usfer}{{\mathbb S}}

\newcommand{\dom}{{\rm dom}\, }
\newcommand{\graph}{{\rm gph}\,}

\newcommand{\nullv}{\mathbf{0}}

\newcommand{\cl}{{\rm cl}\, }

\newcommand{\bd}{{\rm bd}\, }
\newcommand{\inte}{{\rm int}\, }

\newcommand{\Rot}{\mathbf{S}\mathbf{O}}    % Group of the rotations

\newcommand{\SFP}{{\rm SFP}}    % "Split" feasibility problem

\newcommand{\MPSFC}{{\rm MPSFC}}    %  mathematica program with "split" feasibility constraints

\newcommand{\Solv}{\Sigma}    % Solution map (multifunction)

\newcommand{\DRC}{{\rm DRC}}    % Dual regularity condition

\newcommand{\mf}{\nu}    %  merit function

\newcommand{\Lin}{{\mathcal L}}   %   Space of linear bounded operators

\newcommand{\Upsubd}{\widehat{\partial}^+}    % Upper subdifferential

\newcommand{\Fsubd}{\widehat{\partial}}    % Fre\'echet subdifferential

\newcommand{\dcone}[1]{{#1}^{{}^\ominus}}   % negative dual (polar)

\newcommand{\Bder}[2]{{\rm D}_B#1(#2)}           % B-derivative in the sense of S.M. Robinson
\newcommand{\supp}[2]{\sigma\left(#1;#2\right)}      %  support function
\newcommand{\ball}[2]{{\rm B}\left(#1;#2\right)}      %  open ball
\newcommand{\cball}[2]{{\rm B}\left[#1;#2\right]}     % closed ball

\newcommand{\dist}[2]{{\rm dist}\left(#1;#2\right)}
\newcommand{\Tang}[2]{{\rm T}(#1;#2)}        % Bouligand tangent cone
\newcommand{\IDis}[2]{{\rm IT}(#1;#2)}        %  cone of interior displacements
\newcommand{\Ncone}[2]{{\rm N}(#1;#2)}       %  convex analysis normal cone

\newcommand{\Gder}[2]{{\rm D}_{\rm T}#1(#2)}       % contingent derivative
\newcommand{\Ider}[2]{{\rm D}_{\rm IT}#1(#2)}       % interior displacement derivative

\newcommand{\Lip}[3]{{\rm Lip}(#1;#2,#3)}      % costant of Aubin property

\begin{document}

\title{First-order approximations of multifunctions defined implicitly
by ``split" feasibility problems with applications to optimization}

\titlerunning{Approximations of multifunctions defined
by ``split" feasibility problems}        % if too long for running head

%\subtitle{}

\author{Amos Uderzo}

\institute{Amos Uderzo \at
              Department of Mathematics and Applications - Università di Milano-Bicocca \\
             Deutsche Bank Building 1 -
              Milan, MI 20125, Italy\\
              amos.uderzo@unimib.it
}

\date{Received: date / Accepted: date}
%The correct dates will be entered by the editor.

\maketitle

{\hfill Version updated: \today}.

\vskip.5cm

\begin{abstract}
In the present paper, a study is made of first-order
differential properties of certain multifunctions, which
are defined as a solution map associated with a family of parameterized
``split" feasibility problems. The latter are a class
of convex feasibility problems, whose specific structure makes them
suitable for well-established applications in several areas of engineering and systems biology.
As a result, inner and outer approximations of the graphical (contingent) derivative
of the solution map are provided, which are expressed in terms of derivatives
of the problem data.
As an application, a perspective is also discussed for employing some of these achievements
in the formulation of optimality conditions for mathematical programs
with ``split" feasibility constraints
\end{abstract}


%All acknowledgements should be placed in the back of the paper after Conclusions..

\section{Introduction}

In the present paper, a parameterized version of the ``split" feasibility problem
will be considered in the following setting. Let $(\Pb,\|\cdot\|)$, $(\X,\|\cdot\|)$
and $(\Y,\|\cdot\|)$ be real Banach spaces.
Given a mapping $A:\Pb\times\X\longrightarrow\Y$ and two
set-valued maps $C:\Pb\rightrightarrows\X$ and $Q:\Pb\rightrightarrows\Y$,
any element $p\in \Pb$ determines the particular ``split" feasibility problem
$$
  \hbox{find } x\in C(p):\ A(p,x)\in Q(p).   \leqno (\SFP_p)
$$
The solution set-valued map $\Solv:\Pb\rightrightarrows\X$ associated
with the resulting family of parameterized ``split" feasibility problems
$(\SFP_p)$ with $p$ varying in $\Pb$, is therefore given by
\begin{equation}     \label{eq:Solvpsetchar}
  \Solv(p)=\left\{x\in C(p):\ A(p,x)\in Q(p)\right\}=
  C(p)\cap A(p,\cdot)^{-1}(Q(p)).
\end{equation}

Below, some concrete examples of $(\SFP_p)$ taken from different
mathematical contexts are illustrated. They should demonstrate the
versatility of the problem format.

\begin{example}({\it Feasibility under rotation})   \label{ex:SFPpex1}
Let $\Pb=\R$ and $\X=\Y=\R^2$ be endowed with their standard Euclidean structure.
Let $(\Rot(2),\circ)$ denote the group of rotations of $\R^2$ (a.k.a.
special orthogonal group), whose elements are indicated by $O_p$, i.e.
$$
 O_p=\left(\begin{array}{cc}
            \cos (p) & -\sin (p) \\
            \sin (p) & \cos (p)
          \end{array}\right),\quad\forall p\in\R,
$$
and let $A:\R\times\R^2\longrightarrow\R$ be given by
$$
  A(p,x)=O_p x=\left(\begin{array}{c}
            (\cos p)x_1-(\sin p)x_2 \\
            (\sin p)x_1+(\cos p)x_2
          \end{array}\right),\quad p\in\R,
          \ x=(x_1,x_2)\in\R^2.
$$
For every $p\in\R$, $A(p,\cdot)$ is a linear (and bounded) map,
which is invertible with $A(p,\cdot)^{-1}=A(-p,\cdot)\in\Rot(2)$.
Thus, given $C:\R\rightrightarrows\R^2$ and $Q:\R\rightrightarrows\R^2$,
for any $p\in\R$ the corresponding problem $(\SFP_p)$ amounts to finding common points to the
two subsets $C(p)$ and the rotation by $A(p,\cdot)^{-1}$ of $Q(p)$.
By playing with such a geometric interpretation of this class of
$(\SFP_p)$, one can readily obtain solutions maps exhibiting a rich
variety of behaviours.
For example, if $C(p)=|p|\Usfer$ and $Q(p)=(|p|+1)\Usfer$,
for every $p\in\R$, where $\Usfer$ stands for the unit sphere of $\R^2$,
then one sees that
$\Solv:\R\rightrightarrows\R^2$ is constantly empty-valued.

If $C(p)=[0,+\infty)\times\{0\}$ and $Q(p)=p\Usfer$,
for every $p\in\R$, then one sees that $\Solv:\R\rightrightarrows\R^2$ is
single-valued for every $p\in\R$, being given by $\Solv(p)=\{(|p|,0)\}$.

If $C(p)=[0,+\infty)\times\{0\}$ and $Q(p)=(p\Uball)\cap\R^2_+$,
for every $p\in\R$,  where $\Uball$ stands for the unit ball of $\R^2$,
then the solution map $\Solv:\R\rightrightarrows\R^2$
associated to the corresponding family of $(\SFP_p)$ results in
\begin{eqnarray*}
  \Solv(p)=\left\{\begin{array}{ll}
    [0,|p|]\times\{0\}  , & \qquad\hbox{\ if } p\in\bigcup_{k\in\Z}\left[2k\pi,\ 2k\pi+{\pi\over 2}\right],
   \\
    \{(0,0)\}, & \qquad\hbox{\ otherwise,} \\
  \end{array}
  \right.
\end{eqnarray*}
where $\Z$ stands for the integer number set.
So $\Solv$ periodically takes single as well as set values.
\end{example}

\begin{example}({\it Positive solutions to perturbed linear dynamical systems})   \label{ex:SFPpex2}
Given a compact interval $I\subseteq\R$, let $C^1_n(I)$ denote the space
of all vector-valued functions, with each of their $n$ components belonging to $C^1(I)$
and, similarly, let $C_n(I)$ denote the space of all vector-valued functions,
whose $n$ components belong to $C(I)$. Suppose that both these spaces
are equipped with the respective natural norms, making them Banach spaces.
Let $\Pb=C_{n^2+n}(I)$, $\X=C^1_n(I)$ and $\Y=C_n(I)$. Set $p=(P,b)\in C_{n^2+n}(I)$, where
$$
   P(t)=\left(p_{ij}(t)\right)_{i,j=1}^n,\ \hbox{ with } p_{ij}\in C(I),\
   \forall i,j\in\{1,\dots,n\},
   \quad\hbox{ and }\quad
   b\in C_n(I).
$$
Define the (constant) set-valued map $C$ as being
$$
  C(p)=C^1_n(I)^+=\{x\in C^1_n(I):\ x(t)\in\R^n_+,\ \forall t\in I\},
  \quad\forall p\in C_{n^2+n}(I).
$$
and $Q:C_{n^2+n}(I)\rightrightarrows C_n(I)$ as being the projection (single-valued) map
of the space $C_{n^2+n}(I)$ onto the space $C_n(I)$
$$
  Q(p)=Q(P,b)=\{b\}.
$$
Let $A:C_{n^2+n}(I)\times C^1_n(I)\longrightarrow C_n(I)$ be defined by
$$
  A(p,x)=\dot{x}-Px.
$$
The corresponding family of $(\SFP_p)$ amounts to finding,
among the solutions of the perturbed linear ODE system
\begin{equation}   \label{eq:linearODEs}
  \dot{x}(t)=P(t)x(t)+b(t),
\end{equation}
all those having nonnegative components for every $t\in I$.
In this case, for every $p\in C_{n^2+n}(I)$, $\Solv(p)$ is given by the
intersection of a $n$-dimensional affine subspace of $C^1_n(I)$
(the kernel of $A(p,\cdot)$ translated by a particular solution to
the system in (\ref{eq:linearODEs}) ) with the convex cone $C^1_n(I)^+$.
Such kind of problems emerge in various concrete applications of linear dynamical systems,
where the negativity of solutions causes loss of meaning, from the viewpoint of models.
Notice that if the matrix $P$ has constant entries (so $A$ does not
depend on $p$), constructive methods based on the theory of exponential matrices
are at disposal for the explicit resolution of the linear ODE system in (\ref{eq:linearODEs}).
In stark contrast, if the entries of $P$ depend on $t$, the aforementioned
methods fail to work in general and the explicit resolution becomes a
much harder problem.
Adding the sign condition on $x$ represented by $C$ yields a further
complication of the problem, even when the matrix $P$ is constant.
To see this, consider the case in which $I=[0,2\pi]$, $n=2$, and the matrix
$P$ has two complex eigenvalues with vanishing real part. Then, as it possible
to see from the resulting phase portrait of  (\ref{eq:linearODEs}) in $\R^2$, the solvability of
$(\SFP_p)$  dramatically depends on the perturbing term $b$. For instance,
for every $p=(P,\nullv)$ one has $\Solv(p)=\varnothing$.
\end{example}

Of course, that provided by the formula in (\ref{eq:Solvpsetchar}) is
only a formal, as to say implicit representation of $\Solv(p)$.
It is clear that, whenever the map $A(p,\cdot)$ coincides with the
identity map, problem $(\SFP_p)$ reduces to the so-called ($2$-sets) feasibility
problem, which consists in finding an element in the intersection
of two given subsets. In the special case of two closed subspaces
of a Hilbert space, such problem was firstly considered by J. von Neumann
at the beginnings of the 30s and subsequently extended to a pair of general
convex sets by L.M. Br\`egman \cite{Breg65}.
``Split" feasibility problems were introduced later on in \cite{CenElf94}.
Their peculiar ``split" format, where a function appears to link two
subsets in possibly different spaces, thanks to its versatility revealed
a rich potential in theoretical developments and relevant applications.
Thus these problems became soon a specific topic of research (see, for instance,
\cite{Byrn02,Cegi08,CeElKoBo05,HuXuYe25,HuXuYe25b,QuLiu18,ReTrMa20,WaHuLiYa17,WaGaLiYa24}).

The subject of this paper are first-order differential
properties of the multifunction $\Solv$. More precisely,
the investigations exposed focus on representations and approximations of the
contingent derivative of $\Solv$ at
a given point of its graph.
The reader should notice that one can hardly expect to derive the explicit law
determining the multifunction $\Solv$, as being the solution map to a perturbed
family of involved problems, for which constructive resolution methods may lack
(such a difficult should be evident already considering
the Example \ref{ex:SFPpex2}). In contrast,
a deep-rooted scheme in variational analysis, which was successfully applied
to a large variety of problems (e.g. equality/inequality constraint systems,
variational inequalities and complementarity problems, generalized equations,
equilibrium problems, variational conditions) consists in gleaning qualitative and quantitative information
about the respective solution maps by means of their first-order approximations.
Historically, such an approach led to establishing many implicit multifunction theorems,
suitable for several different contexts (see, among the others,
\cite{ArMoZh23,BorZhu05,DonRoc14,DurStr12,HuNgTh13,LedZhu99,LeTaYe08,LeeYen11,LevRoc94,Levy96,Mord94,Mord06,Robi91}).
Typically, a key element appearing in most of the aforementioned results
is some representation of the derivative (or coderivative) of the solution
map in terms of problem data. This constructive achievement can be
regarded as the {\it punctum saliens} within the so-called sensitivity
analysis of a certain problem.
The quantitative information about the behaviour of the solution map
that is possible to obtain from its first-order approximations turn out
to play an essential role in several contexts, such as robust approaches
to the problem at the issue, value-function/marginal analysis, optimality conditions.
This motivates the present investigations with reference to ``split" feasibility
problems. In synthesis, the aim of the present paper is to complement
the perturbation analysis of ``split" feasibility problems
initiated in \cite{HuXuYe25,HuXuYe25b,Uder26} by providing
insights into a fundamental derivative-like object able to
describe the local geometry of $\Solv$ near a given solution.

%%%%%%%%%%%%%%%%%%%%%%%%%%%%%%%%%%%

The contents of the paper are arranged as follows.
In Section \ref{Sect:2}, the basic notations, standing assumptions
and needed technical preliminaries,
mainly tools from convex and set-valued analysis, are gathered.
Sections \ref{Sect:3} contains the main results established in the paper,
namely first-order approximations from both inside and outside
(in some case, an exact representation) of the contingent derivative of $\Solv$.
In Section \ref{Sect:4}, some of the above results are applied
to deriving necessary optimality conditions in the context of mathematical
programming with ``split" feasibility constraints.

\vskip.5cm

%%%%%%%%%%%%%%%%%%%%%%%%%%%%%%%%%%%%%%%%%%%%%%%%%%%%%%%%%%%%%%%%%%%%%%%%%%%%%%%%%%%%%%%%%%%%%%%%%%%%%%%

\section{Preliminaries}     \label{Sect:2}

\subsection{\bf Basic notations and standing assumptions} \label{Sect:2.1}

The basic notations in use throughout the paper are standard.
$\R$ denotes the field of real numbers, $\N$ denotes its subset of
the natural ones, while $\R^d$ stands for the $d$-dimensional Euclidean space.
Given a (extended) scalar function $\varphi:\X\longrightarrow\R\cup\{\pm\infty\}$
defined on a (real) Banach space $\X$,
$\dom\varphi=\varphi^{-1}(\R)$ denotes its domain and, if $\alpha\in\R$,
$[\varphi\le\alpha]=\{x\in\X\ :\ \varphi(x)\le\alpha\}$ denotes its $\alpha$-sublevel
set, $[\varphi>\alpha]=\{x\in\X:\ \varphi(x)>\alpha\}$
denotes its $\alpha$-strict superlevel set, whereas $[\varphi=\alpha]=\varphi^{-1}(\alpha)$
the $\alpha$-level set.
If $\Omega$ is a subset of $\X$,
$\dist{x}{\Omega}=\inf_{z\in\Omega}\|x-z\|$ denotes that distance of a point $x\in\X$ from
$\Omega$, with the convention that $\dist{x}{\varnothing}=+\infty$.
Consistently, if $r\ge 0$, $\cball{\Omega}{r}=[\dist{\cdot}{\Omega}\le r]$
indicates the closed $r$-enlargement of the set $\Omega$, with radius $r$. In particular, if
$\Omega=\{x\}$, $\cball{x}{r}$ denotes the closed ball with center $x$
and radius $r$.
Similarly, if $r>0$, $\ball{\Omega}{r}=[\dist{\cdot}{\Omega}<r]$ and
$\ball{x}{r}$ stand for the open $r$-enlargement of $\Omega$ and the open ball centered
at $x$, with radius $r$.
The symbols $\cl\Omega$, $\inte\Omega$, and $\bd\Omega$ indicate the topological closure,
the boundary and the interior of $\Omega$, respectively.
The symbol $\nullv$ stands for the null element of any Banach space.
The (topological) dual space of $\X$ is denoted by $\X^*$, with $\langle\cdot,\cdot\rangle
:\X^*\times\X\longrightarrow\R$ denoting their duality pairing and $\nullv^*$ standing for
the zero functional over $\X$. In this context, $\cball{\nullv}{1}$ and $\bd\cball{\nullv}{1}$
(resp. $\cball{\nullv^*}{1}$ and $\bd\cball{\nullv^*}{1}$)
will be simply indicated by $\Uball$ and $\Usfer$ (resp. $\Uball^*$ and $\Usfer^*$ ), respectively.
By $\Lin(\X,\Y)$ the space of all bounded linear operators between $\X$ and $\Y$
is denoted. This space will be equipped with the operator norm, indicated by $\|\cdot\|_{\Lin}$.
Given $A\in\Lin(\X,\Y)$, the map $A^*:\Y^*\longrightarrow\X^*$ indicates the adjoint
operator to $A$, belonging to $\Lin(\Y^*,\X^*)$.
Given a set-valued map $F:\Pb\rightrightarrows \X$ between Banach spaces, $\dom F=
\{p\in \Pb:\ F(p)\ne\varnothing\}$ denotes the effective domain
of $F$, while $\graph F=\{(p,x)\in \Pb\times    X:\ x\in F(p)\}$ its graph.

The acronyms l.s.c. and u.s.c., standing for lower semicontinuous
and upper semicontinuous, respectively, will be adopted whenever needed.

Given a cone $K\subseteq\X$, by $\dcone{K}=\{x^*\in\X^*: \langle x^*,x\rangle
\le 0,\quad\forall x\in K\}$ the dual (negative) cone to $K$ is denoted.
The following calculus rule concerning the dual cone will be used
in the sequel: given $A\in\Lin(\X,\Y)$ and two closed and convex cones
$K\subseteq\X$ and $H\subseteq\Y$, it holds
\begin{equation}   \label{eq:dconecalrule}
  \dcone{\left[K\cap A^{-1}(H)\right]}=\cl\left[\dcone{K}+A^*(\dcone{H})\right]
\end{equation}
(see, for instance, \cite[Lemma 2.4.1]{Schi07}).

Given a nonempty set $\Omega\subseteq\X$, the function $\supp{\cdot}{\Omega}:\X^*\longrightarrow
\R\cup\{+\infty\}$, defined by $\supp{x^*}{\Omega}=
\sup_{x\in\Omega}\langle x^*,x\rangle$, is known in convex analysis as the support
function of $\Omega$.
In Section \ref{Sect:4}, an exact relation ruling
the behaviour of the support function with respect to the intersection of sets
will be needed. Given two closed and convex subsets $\Omega_1$ and $\Omega_2$
of the same Banach space $\X$, suppose that one of the following conditions
is satisfied:
\begin{itemize}
  \item[(i)] $\inte(\Omega_1)\cap\Omega_2\ne\varnothing$;

  \item[(ii)] $\nullv\in\inte(\Omega_1-\Omega_2)$ and $\Omega_2$ is bounded;
\end{itemize}
then the below representation holds
\begin{eqnarray}   \label{eq:suppinter}
  \supp{x^*}{\Omega_1\cap\Omega_2} & = & \supp{x^*}{\Omega_1}\Box\supp{x^*}{\Omega_2} \\
  &=& \inf\{\supp{x_1^*}{\Omega_1}+\supp{x_2^*}{\Omega_2}:\ x^*_1+x_2^*=x^*\},
  \quad\forall x^*\in\X^*,  \nonumber
\end{eqnarray}
where $\Box$ stands for the (infimal) convolution of convex functions
(see \cite[Theorem 4.23]{MorNam22}).

\begin{remark}    \label{rem:concsupF}
It well-known that, if $\Omega\subseteq\Pb$ is any set,
its support function $x^*\mapsto\supp{x^*}{\Omega}$ is a always sublinear,
i.e. positively homogeneous and convex (for more details, see \cite[Section 4.1.2]{MorNam22}).
In view of the subsequent analysis, it is useful to observe also that,
if $F:\Pb\rightrightarrows\X$ is a convex multifunction between Banach spaces,
then the function defined by means of the support function and any fixed $x^*\in\Pb^*$ as
\begin{equation}    \label{map:defsuppFsv}
  u\mapsto \supp{x^*}{F(u)},
\end{equation}
turns out to be concave on $\Pb$. Moreover, whenever $F$ is positively homogeneous,
i.e. $F(tu)=tF(u)$, for every $t>0$ and $u\in\Pb$, then also the function defined
in $(\ref{map:defsuppFsv})$ is positively homogeneous.
\end{remark}

The meaning of some further symbols employed in the subsequent sections
will be explained contextually to their introduction.

%%%%%%%%%%%%%%%%%%%%%%%%%%%%%%%%%%%%%%%%%%%

Throughout the paper, with reference to a family $(\SFP_p)$, the following assumptions on the problem data
are maintained:

\begin{itemize}

\item[($a_1$)] the set-valued maps $C:\Pb\rightrightarrows\X$ and $Q:\Pb\rightrightarrows\Y$ take
nonempty, closed and convex values;

\item[($a_2$)] $A(p,\cdot)\in\Lin(\X,\Y)$, for every $p\in\Pb$.

\end{itemize}

As a comment, notice that the above assumptions mean that, for any value of the parameter $p\in\Pb$,
the corresponding instance of $(\SFP_p)$ satisfies the standard assumptions of
the original ``split" feasibility problem (see \cite{CenElf94}). On the other hand, the assumption ($a_2$)
allows for a possibly nonlinear parameter dependence of $A(p,\cdot)$, as it
happens for the instance of $(\SFP_p)$ considered in Example \ref{ex:SFPpex1}.

%%%%%%%%%%%%%%%%%%%%%%%%%%%%%%%%%%%%%

\subsection{\bf Local cone approximations of sets and graphical derivatives
of multifunctions}   \label{Sect:2.2}

A subset of a Banach space can be locally approximated by a cone
near one of its points in several ways. For the purposes of the
present analysis, the following two notions of cone approximation
have to be recalled. Given a set $\Omega\subseteq\X$ and
$x_0\in\Omega$, the set
$$
  \Tang{\Omega}{x_0}=\left\{v\in\X:\
  \forall\delta>0,\ \exists v'\in\ball{v}{\delta},\ \exists t\in (0,\delta):\
  x_0+tv'\in\Omega \right\}
$$
is called Bouligand tangent (contingent) cone to $\Omega$ at $x_0\in\Omega$;
the set
$$
  \IDis{\Omega}{x_0}=\left\{v\in\X:\
  \exists\delta>0:\ \forall v'\in\ball{v}{\delta},\ \forall t\in (0,\delta),\
  x_0+tv'\in\Omega \right\}
$$
is called cone of interior displacements to $\Omega$ at $x_0\in\Omega$.
Of course, both of the above sets are actually cones. Whereas the contingent
cone is known to be always closed and nonempty, the cone of interior
displacement may happen to be empty and typically fails to be closed.
Moreover, $\Tang{\Omega}{x_0}$ is convex if $\Omega$ does the same.
By relying on set-convergence concepts, the former notion can be equivalently
defined as an upper limit of a proper quotient involving the set $\Omega$
(see, for more details, \cite{AubFra90,DonRoc14,RocWet98}).

By exploiting the distance function $x\mapsto\dist{x}{\Omega}$ associated with
a given subset $\Omega$ of a Banach space, the following characterization of the Bouligand
tangent cone can be formulated in terms of generalized
derivative (see \cite[Definition 4.1.1]{AubFra90})
$$
  \Tang{\Omega}{x_0}=\left\{v\in\X:\
  \liminf_{t\downarrow 0}\frac{\dist{x_0+tv}{\Omega}}{t}=0\right\}.
$$

Subsets of a Banach space can be also locally approximated near one of their points
by cones in the dual space. Following a fundamental approach in convex analysis,
given $\Omega\subseteq\X$ and $x_0\in\Omega$, the convex cone
$$
   \Ncone{\Omega}{x_0}=\left\{x^*\in\X^*:\ \langle x^*,x-x_0\rangle\le 0,
   \quad\forall x\in\Omega\right\}.
$$
is called normal cone to $\Omega$ at $x_0\in\Omega$.

As in the case of single-valued maps, derivatives of set-valued maps
locally approximate them near a given point. Several approaches have
been proposed to the differentiation of multifunctions. The graphical
approach relies on local cone approximation in the product space
of the graph of a given multifunction. More formally,
given $F:\X\rightrightarrows\Y$ and $(x_0,y_0)\in\graph F$, the
contingent (graphical) derivative of $F$ at $(x_0,y_0)$ is the set-valued map
$\Gder{F}{x_0,y_0}:\X\rightrightarrows\Y$ defined via the cone approximation
of its graph as
$$
  \graph\Gder{F}{x_0,y_0}=\Tang{\graph F}{(x_0,y_0)}.
$$
Analogously, the
interior-displacement (graphical) derivative of $F$ at $(x_0,y_0)$ is the set-valued map
$\Ider{F}{x_0,y_0}:\X\rightrightarrows\Y$ defined via the cone approximation
of its graph as
$$
  \graph\Ider{F}{x_0,y_0}=\IDis{\graph F}{(x_0,y_0)}.
$$
From the very definition it is clear that $\Gder{F}{x_0,y_0}$ is a closed
and positively homogeneous multifunction.
The next lemma provides descriptions of elements in the above graphical
derivatives, which will be useful in what follows.

\begin{lemma}    \label{lem:grphderdist}
Given a set-valued map $F:\X\rightrightarrows\Y$ between normed spaces
and $(x_0,y_0)\in\graph F$, then for every $u\in\X$
\begin{itemize}
  \item[(i)] \begin{equation}     \label{in:Gderscarep}
  v\in \Gder{F}{x_0,y_0}(u) \quad\hbox{ iff }\quad
  \liminf_{u' \to u \atop t\downarrow 0} \frac{\dist{y_0+tv}{F(x_0+tu')}}{t}=0.
\end{equation}
  \item[(ii)] \begin{equation}     \label{in:Iderincl}
  \quad\hbox{ if } v\in \Ider{F}{x_0,y_0}(u) \quad\hbox{ then }\quad
  \lim_{u' \to u \atop t\downarrow 0} \frac{\dist{y_0+tv}{F(x_0+tu')}}{t}=0.
\end{equation}
\end{itemize}
\end{lemma}

\begin{proof}
(i) It suffices to observe that, by elementary properties of the distance
function induced by a norm, it holds
$$
  \frac{\dist{y_0+tv}{F(x_0+tu')}}{t}=\dist{v}{\frac{F(x_0+tu')-y_0}{t}},
  \quad\forall u'\in\X,\ \forall t>0,
$$
and then to apply the characterization in \cite[Proposition 5.1.4]{AubFra90}.

(ii) Take an arbitrary $v\in \Ider{F}{x_0,y_0}(u)$. According to the definition of
interior-displacement derivative, one has $(u,v)\in\IDis{\graph F}{(x_0,y_0)}$.
This fact means that there exists $\delta>0$ such that
\begin{eqnarray*}
  (x_0,y_0)+t(u',v') &=& (x_0+tu',y_0+tv') \in \graph F,  \\
  & & \forall (u',v')\in\ball{u}{\delta}\times\ball{v}{\delta},\
  \forall t\in (0,\delta).
\end{eqnarray*}
From the last inclusion, in particular, it follows
$$
  y_0+tv\in F(x_0+tu'),\quad\forall (u',t)\in \ball{u}{\delta}\times (0,\delta),
$$
whence one can derive immediately
$$
\lim_{u' \to u \atop t\downarrow 0} \frac{\dist{y_0+tv}{F(x_0+tu')}}{t}=0.
$$
\smartqed\qed
\end{proof}

Let $\Omega\subseteq\Y$ be a closed convex set, let $y_0\in\Omega$
and let $A\in\Lin(\X,\Y)$ be such that $\nullv\in\inte\left[A(\X)-\Omega\right]$.
Then, for each $x_0\in A^{-1}(y_0)$ the following representation is known to hold
\begin{equation}   \label{eq:A-1Tang}
  \Tang{A^{-1}(\Omega)}{x_0}=A^{-1}(\Tang{\Omega}{y_0})
\end{equation}
(see, for instance, \cite[Chapter 4.2]{AubFra90}). Notice that the above qualification
condition takes place, in particular, whenever $A$ is onto.

\begin{lemma}   \label{lem:A-1Q}
Let $Q:\Pb\rightrightarrows\Y$ be a closed convex multifunction and
let $(p_0,y_0)\in\graph Q$. If $A\in\Lin(\X,\Y)$ is onto and $x_0\in
A^{-1}(y_0)$, the set-valued map $A^{-1}\circ Q:\Pb\rightrightarrows\X$
is closed and convex and it holds
$$
  \Gder{(A^{-1}\circ Q)}{p_0,x_0}=A^{-1}\circ \Gder{Q}{p_0,y_0},
$$
where $A^{-1}:\Y\rightrightarrows\X$ denotes the inverse multifunction to
$A$.
\smartqed\qed
\end{lemma}

\begin{proof}
Observe that in the current setting $A^{-1}:\Y\rightrightarrows\X$ is
a closed convex process. Therefore, as a composition of convex multifunctions,
$A^{-1}\circ Q$ is convex. Moreover, as $A$ is continuous and $Q$ is closed,
$A^{-1}\circ Q$ turns out to be closed. Therefore, $\graph(A^{-1}\circ Q)$
is a closed convex subset of $\Pb\times\X$.

Now, consider the map $\mathcal{A}:\Pb\times\X\longrightarrow\Pb\times\Y$,
defined by
$$
  \mathcal{A}(p,x)=(p,A(x)).
$$
This map is linear and bounded, because so is $A$, namely $\mathcal{A}\in\Lin(\Pb\times\X,\Pb\times\Y)$.
Moreover, $\mathcal{A}$ is onto, because $A$ is onto. Observe that
\begin{equation}     \label{eq:graphA-1graphQ}
  \graph(A^{-1}\circ Q)=\mathcal{A}^{-1}(\graph Q).
\end{equation}
Indeed, if $(p,x)\in\graph(A^{-1}\circ Q)$, then $A(x)\in Q(p)$, so that
$(p,A(x))\in\graph Q$, and hence, according to the definition of $\mathcal{A}$,
it is $(p,x)\in \mathcal{A}^{-1}(\graph Q)$,
while all these implications can be reverted.
By employing the representation in (\ref{eq:A-1Tang}), which can be applied
here as $\graph Q$ is closed and convex, $\mathcal{A}\in\Lin(\Pb\times\X,\Pb\times\Y)$
is onto, so $(\nullv,\nullv)\in\inte\bigl[\mathcal{A}(\Pb\times\X)-\graph Q\bigl]$,
on account of the equality in (\ref{eq:graphA-1graphQ}) one obtains
\begin{eqnarray*}
  \graph\left(\Gder{(A^{-1}\circ Q)}{p_0,x_0}\right) &=& \Tang{\graph(A^{-1}\circ Q)}{(p_0,x_0)} \\
   &=& \Tang{ \mathcal{A}^{-1}(\graph Q)}{(p_0,x_0)})  \\
   &=& \mathcal{A}^{-1}(\Tang{\graph Q}{(p_0,y_0)}) \\
   &=& \mathcal{A}^{-1}(\graph\Gder{Q}{p_0,y_0})  \\
   &=& \graph (A^{-1}\circ \Gder{Q}{p_0,y_0}),
\end{eqnarray*}
which proves the equality in the thesis.
\smartqed\qed
\end{proof}

%%%%%%%%%%%%%%%%%%%%%%%%%%%%%%%%%%%%%%%%%%%%%%%%%%

\subsection{\bf Dual approximations of scalar functions}     \label{Sect:2.3}

Let $\varphi:\X\longrightarrow\R\cup\{\pm\infty\}$ be a function defined on
a Banach space and let $x_0\in\dom\varphi$. Following \cite[Chapter 1.3]{Mord06},
the subset of $\X^*$
$$
   \Upsubd\varphi(x_0)=\left\{v^*\in\X^*:\ \limsup_{x\to x_0}
   {\varphi(x)-\varphi(x_0)-\langle v^*,x-x_0\rangle\over\|x-x_0\|}
   \le 0\right\}
$$
is called (Fr\'echet) upper subdifferential of $\varphi$ at $x_0$.
It is useful to recall that if $\varphi$ is concave, then
$\Upsubd\varphi(x_0)$ coincides with the superdifferential (a.k.a. upper subdifferential)
of $\varphi$ at $x_0$ in the sense of
convex analysis. If $\varphi$ is Fr\'echet differentiable at $x_0$, then $\Upsubd\varphi(x_0)=
\{\nabla\varphi(x_0)\}$. In general, the upper subdifferential
of a function $\varphi$ at $x_0$ is connected with the Fr\'echet subdifferential
of $-\varphi$ at $x_0$ (here denoted by $\Fsubd(-\varphi)(x_0)$) by the following relation
$$
   \Upsubd\varphi(x_0)=-\Fsubd\left(-\varphi\right)(x_0).
$$
Such a relation enables one to derive from \cite[Theorem 1.88(i)]{Mord06}
the following variational description of $\Upsubd\varphi(x_0)$ that
will be useful in the sequel:
for every $v^*\in\Upsubd\varphi(x_0)$ there is a function
$\varsigma:\X\longrightarrow\R$, with $\varsigma(x_0)=\varphi(x_0)$
and $\varsigma(x)\ge \varphi(x)$ for every $x\in\X$, such that $\varsigma$
is Fr\'echet differentiable at $x_0$ with $\nabla\varsigma(x_0)=v^*$.

%%%%%%%%%%%%%%%%%%%%%%%%%%%%%%%%%%%%%%%%%%%%%%%%%%%%%

\subsection{\bf Error bounds for parameterized ``split" feasibility problems} \label{Sect:2.4}

The main result of the paper will be established by employing
an error bound estimate for $(\SFP_p)$. Estimates of this kind often requires
a qualification condition to hold.
In order to formulate a qualification condition, which is suitable to the present
context, one needs to introduce the following quantities, which are
expressed in terms of normal cone constructions.
First of all, define the residual function $\mf:\Pb\times\X\longrightarrow [0,+\infty)$
associated to $(\SFP_p)$ as being
$$
  \mf(p,x)=\dist{A(p,x)}{Q(p)}+\dist{x}{C(p)}.
$$
It is clear that $\mf$ measures the violation for $x$ to fulfil
both the conditions $x\in\ C(p)$ and $A(p,x)\in\ Q(p)$.
More precisely, the standing assumption $(a_1)$ ensures that
$x\in\Solv(p)$ iff $\mf(p,x)=0$.
Then, given any $\delta>0$ and any pair $(p,x)\in \Pb\times\X$, set by convenience $\mf_{A,Q}(p,x)=\dist{A(p,x)}{Q(p)}$,
$\mf_C(p,x)=\dist{x}{C(p)}$, $d_{A,Q}=\dist{A(p,x)}{Q(p)}$ and $d_C=\dist{x}{C(p)}$.
On the base of these notations, define
\begin{eqnarray}   \label{eq:deftauAQ}
    \tau_{A,Q}(\delta)=\inf\bigg\{\|x^*\|:\ & x^*\in A(p,\cdot)^*\left(\Ncone{A(p,x)}{\cball{Q(p)}{d_{A,Q}})}\cap\Usfer^*\right) \nonumber \\
     & +[\Ncone{x}{\cball{C(p)}{d_C}}\cap\Uball^*],  \nonumber \\
    & (p,x)\in [\mf_{A,Q}>0]\cap\left[\ball{\bar p}{\delta}\times\ball{\bar x}{\delta}\right]\bigg\}
\end{eqnarray}
and
\begin{eqnarray}   \label{eq:deftauC}
   \tau_C(\delta)=\inf\bigg\{ & \|x^*\|:\  x^*\in A(p,\cdot)^*\left(\Ncone{A(p,x)}{Q(p)}\cap\Uball^*\right)
      +[\Ncone{x}{\cball{C(p)}{d_C}}\cap\Usfer^*],  \nonumber \\
    & (p,x)\in [\mf_{A,Q}=0]\cap[\mf_{C}>0]\cap\left[\ball{\bar p}{\delta}\times\ball{\bar x}{\delta}\right]\bigg\}.
\end{eqnarray}

In the above setting, a family $(\SFP_p)$ is said to satisfy the {\it dual regularity condition}
(for short, \DRC) near $(\bar p,\bar x) \in\graph\Solv$  if
$$
  \exists\delta>0:\ \tau(\delta)=\min\{\tau_{A,Q}(\delta),\,
  \tau_C(\delta)\}>0.   \leqno(\DRC)
$$

\vskip.5cm

\begin{theorem}(\cite[Theorem 3.1]{Uder26})    \label{thm:Solverbo}
With reference to a family $(\SFP_p)$, let $(\bar p,\bar x)\in
\graph\Solv$. Suppose that:
\begin{itemize}
  \item[(i)]  $C:\Pb\rightrightarrows\X$ is l.s.c. at $(\bar p,\bar x)$;

  \item[(ii)] $Q:\Pb\rightrightarrows\Y$ is l.s.c. at $\left(\bar p,A(\bar p,\bar x)\right)$;

  \item[(iii)] $A(\cdot,\bar x):\Pb\longrightarrow\Y$ is continuous at $\bar p\in\Pb$;

  \item[(iv)] the qualification condition $(\DRC)$ holds near $(\bar p,\bar x)$.

\end{itemize}
Then, there exist $\eta,\, \zeta>0$ such that:
\begin{itemize}
  \item[(t)] $\Solv(p)\cap\cball{\bar x}{\eta}\ne\varnothing,
  \quad\forall p\in\cball{\bar p}{\zeta}$;

  \item[(tt)] the following estimate holds
\begin{equation}    \label{in:serbo}
  \dist{x}{\Solv(p)}\le\tau(\delta)^{-1}\mf(p,x),\quad\forall (p,x)\in
  \cball{\bar p}{\zeta}\times\cball{\bar x}{\eta}.
\end{equation}
\end{itemize}
\end{theorem}

\begin{remark}
Notice that as a byproduct of Theorem \ref{thm:Solverbo}, one obtains
the local solvability of $(\SFP_p)$ for every $p$ in a neighbourhood of $\bar p$,
along with the lower semicontinuity of $\Solv$ at $(\bar p,\bar x)$.
In particular, $\bar p\in\inte(\dom\Solv)$.
\end{remark}

\vskip.5cm

%%%%%%%%%%%%%%%%%%%%%%%%%%%%%%%%%%%%%%%%%%%%%%%%%%%%%%%%%%%%%%%%%%%%%%%%%%%%%%%%%%%%%%%%%%%%%%%%%%%%%%%

\section{First-order approximations of the solution map}   \label{Sect:3}

When dealing with cone approximation of the solution map $\Solv$, defined
as in (\ref{eq:Solvpsetchar}), a first question to be faced
is how to represent tangential approximations of each value of $\Solv$
in terms of problem data. This question can be readily addressed by standard tools
of convex analysis. Indeed, by exploiting well-known calculus rules for
the contigent cone with respect to set operations (in particular, intersection) and
to taking the counterimage of a set, under a standard qualification condition for smooth
single-valued maps (known as Robinson condition, see \cite[Section 11.4]{Schi07}),
it is possible to establish immediately the following result.

\begin{proposition}[Tangential approximation of images]   \label{pro:appval}
Given $(\bar p,\bar x)\in\graph\Solv$, if
\begin{equation}   \label{in:Robiqual}
    \nullv\in\inte[A(\bar p,\cdot)(C(\bar p))-Q(\bar p)],
\end{equation}
then it holds
\begin{equation}    \label{eq:Solvimapprox}
  \Tang{\Solv(\bar p)}{\bar x}=\Tang{C(\bar p)}{\bar x}\cap
  \bigl[A(\bar p,\cdot)^{-1}\left(\Tang{Q(\bar p)}{A(\bar p,\bar x)}\right)\bigl].
\end{equation}
\end{proposition}

\begin{proof}
It suffices to remember that, according to (\ref{eq:Solvpsetchar}),
$\Solv(\bar p)= C(\bar p)\cap A(\bar p,\cdot)^{-1}(Q(\bar p))$ and
to observe that, as $A(\bar p,\cdot)$ is linear and bounded according to the assumption
$(a_2)$, then the Robinson condition becomes
\begin{eqnarray*}
  \nullv &\in &\inte[A(\bar p,\bar x)+\nabla A(\bar p,\cdot)(C(\bar p))-\bar x)-Q(\bar p)]   \\
  &= &\inte[A(\bar p,\cdot)(C(\bar p))-Q(\bar p)].
\end{eqnarray*}
Then the representation in (\ref{eq:Solvimapprox}) follows at once by virtue of
well-known calculus rules for the contingent cone (see, for instance,
\cite[Theorem 11.4.4]{Schi07}), under the proper qualification condition.
\smartqed\qed
\end{proof}

\begin{remark}
Under the standing assumption $(a_1)$ and $(a_2)$, it is clear that
$A(\bar p,\cdot)(C(\bar p))$ and $Q(\bar p)$ are convex subsets of $\Y$, while $Q(\bar p)$ is
closed. Whenever $A(\bar p,\cdot)(C(\bar p))$ is also closed,
the Robinson condition in (\ref{in:Robiqual}) is known
to be sufficient for the pair $A(\bar p,\cdot)(C(\bar p))$ and $Q(\bar p)$
to be subtransversal at the common point $A(\bar p,\bar x)$. In other words, the hypothesis
of Proposition \ref{pro:appval} can be geometrically viewed as a condition about
a proper mutual arrangement of the two above sets near $A(\bar p,\bar x)$. Since
such an arrangement can be reformulated in purely metric terms as the existence
of positive $\kappa$ and $\delta$ such that
\begin{eqnarray*}
 \dist{y}{A(\bar p,\cdot)(C(\bar p))\cap Q(\bar p)}
 &\le& \kappa \max\left\{\dist{y}{A(\bar p,\cdot)(C(\bar p))},\
 \dist{y}{Q(\bar p)}\right\},  \\
 & & \forall y\in\cball{A(\bar p,\bar x)}{\delta},
\end{eqnarray*}
then the hypothesis in Proposition \ref{pro:appval} requires
roughly speaking that ``if you are close to both the
sets of the pair, then the intersection cannot be too far away"
to quote \cite{BauBor96}, where more discussions about subtransversality
of pair of sets and their role in the convergence analysis of iterative methods for solving
feasibility problems can be found. For connections of this property with
error bounds, weak sharp minimality, linear subopenness/metric subregularity,
the reader is referred to \cite{KrLuTh17}.
\end{remark}

Of course, representing the tangential approximation of the values of the
multifunction $\Solv$ one-by-one does not amount to provide a first-order
approximation of the set-valued map $\Solv$ itself, as resulting from its graphical differentiation.
This because the graphical derivative of the multifunction $\Solv$ is affected by the behaviour
of $\Solv$ as both $p$ and $x$ vary near the reference pair $(\bar p,\bar x)$,
whereas the tangential approximation of the specific value $\Solv(\bar p)$ takes
into account only the local geometry of the set $\Solv(\bar p)$ in the
space $\X$, while $p$ remaining fixed.

Nevertheless, in  the particular case, in which $C$ and $Q$ are closed and convex
multifunctions, and $A$ does not depend on $p$, i.e. $A\in\Lin(\X,\Y)$,
and is onto, following an approach similar to that of Proposition \ref{pro:appval}
leads to an exact representation of $\Gder{\Solv}{\bar p,\bar x}$.

\begin{theorem}[Exact representation of ${\rm D}_{\rm T}\Solv$]    \label{thm:convspecase}
With reference to a family $(\SFP_p)$, let $(\bar p,\bar x)\in\graph\Solv$.
Suppose that:
\begin{itemize}
  \item[(i)] $C$ and $Q$ are closed convex multifunctions;

  \item[(ii)] $A\in\Lin(\X,\Y)$ is onto;

  \item[(iii)] $\nullv\in\inte[\graph C-\graph(A^{-1}\circ Q)]$.
\end{itemize}
Then, $\Solv:\Pb\rightrightarrows\X$ is a closed and convex multifunction
and it holds
\begin{equation}   \label{GderSolvexact}
  \Gder{\Solv}{\bar p,\bar x}(u)=\Gder{C}{\bar p,\bar x}(u) \cap
  A^{-1}\left( \Gder{Q}{\bar p,A(\bar x))}(u)\right),
  \quad\forall u\in\Pb.
\end{equation}
\end{theorem}

\begin{proof}
By hypothesis (ii), $\Solv$ can be expressed through its graph as being
\begin{equation}\label{eq:graphSolvCAQ}
  \graph\Solv=\graph C\cap \graph (A^{-1}\circ Q).
\end{equation}
According to Lemma \ref{lem:A-1Q}, by the hypotheses (i) and (ii)
$\graph(A^{-1}\circ Q)$ is closed and convex, as well as $\graph C$.
Thus $\graph\Solv$ is closed and convex. From the equality in (\ref{eq:graphSolvCAQ})
it follows
$$
  \Tang{\graph\Solv}{(\bar p,\bar x)}=\Tang{\graph C\cap \graph (A^{-1}\circ Q)}{(\bar p,\bar x)}.
$$
By virtue of the well-known calculus rule for the Bouligand tangent cone
of intersection of closed convex sets, which is valid under the qualification
in hypothesis (iii), one obtains
\begin{eqnarray*}
  \graph \Gder{\Solv}{\bar p,\bar x} &=& \Tang{\graph\Solv}{(\bar p,\bar x)}=\Tang{\graph C}{(\bar p,\bar x)}\cap
  \Tang{\graph (A^{-1}\circ Q)}{(\bar p,\bar x)} \\
   &=& \graph\Gder{C}{\bar p,\bar x}\cap \graph \Gder{(A^{-1}\circ Q)}{\bar p,\bar x}.
\end{eqnarray*}
By applying the chain rule for the contingent derivative of composition
of multifunctions provided in Lemma \ref{lem:A-1Q}, one finds
\begin{equation}    \label{eq:protodiffSolv}
  \graph \Gder{\Solv}{\bar p,\bar x}=\graph\Gder{C}{\bar p,\bar x}\cap
  \graph\left(A^{-1}\circ \Gder{Q}{\bar p,A(\bar x)}\right),
\end{equation}
which entails the equality in the thesis.
\smartqed\qed
\end{proof}

\begin{remark}({\it Proto-differentiability of $\Solv$})
It is worth pointing out that, in the special setting of Theorem \ref{thm:convspecase},
$\Solv$ gains an enhanced differential property called proto-differentiability.
Indeed, from the equality in (\ref{eq:protodiffSolv}) it is clear that $\Gder{\Solv}{\bar p,\bar x}$
is a closed sublinear multifunction (a.k.a. closed convex process).
Since it is closed, $\Gder{\Solv}{\bar p,\bar x}$ is outer semicontinuous 
at each $u\in\Pb$ (see \cite[Theorem 3B.2(c)]{DonRoc14}). According to \cite[Example 8.39]{RocWet98},
the combination of convexity and outer semicontinuity yields the property of
graph regularity  (see \cite[Definition 8.38]{RocWet98}), which in turn is known to be a sufficient condition for the
proto-differentiability (see \cite[Proposition 8.41]{RocWet98}).
The latter property means that for each $u\in\Pb$ and $v\in\Gder{\Solv}{\bar p,\bar x}(u)$,
for any choice of $(t_n)_{n\in\N}$ in $(0,+\infty)$, with $t_n\downarrow 0$ as $n\to\infty$,
there must exist sequences $(u_n)_{n\in\N}$ in $\Pb$ and $(v_n)_{n\in\N}$ in $\X$, such that
$\bar x+t_nv_n\in\Solv(\bar p+t_nu_n)$ for every $n\in\N$ and $(u_n,v_n)\to (\bar p,\bar x)$,
as $n\to\infty$. In terms of set-convergence of quotients, the $\limsup$ defining the notion
of contingent cone becomes in this case a full limit.
For more details about this property and the related notion of graph regularity, the reader
is referred to \cite[Section 8.H${}^*$]{RocWet98}.

\end{remark}

Next, when passing to consider the more general case of an arbitrary class $(\SFP_p)$,
one can hardly expect to obtain an exact representation as in Theorem
\ref{thm:convspecase}. Under rather general assumptions, a first result
can be established, which provides an approximated estimate of ${\rm D}_{\rm T}\Solv$
from outside.

\begin{theorem}[Outer approximation of ${\rm D}_{\rm T}\Solv$]    \label{thm:outapproxDSolv}
With reference to a family $(\SFP_p)$, let $(\bar p,\bar x)\in\graph\Solv$.
If $A$ is Fr\'echet differentiable at $(\bar p,\bar x)$, then the following
inclusion holds
\begin{equation}   \label{GderSolvout}
  \Gder{\Solv}{\bar p,\bar x}(u)\subseteq \Gder{C}{\bar p,\bar x}(u) \cap
  \bigl[\nabla A(\bar p,\bar x)(u,\cdot)^{-1}\left(\Gder{Q}{\bar p,A(\bar p,\bar x)}(u)\right)\bigl],
  \quad\forall u\in\Pb.
\end{equation}
\end{theorem}

\begin{proof}
Fix an arbitrary $u\in\Pb$. If $\Gder{\Solv}{\bar p,\bar x}(u)=\varnothing$, then
the inclusion in (\ref{GderSolvout}) becomes obvious. Otherwise, suppose that $v\in
\Gder{\Solv}{\bar p,\bar x}(u)$, which means that $(u,v)\in\Tang{\graph\Solv}{(\bar p,\bar x)}$.
Then, there exist $(u_n,v_n)_{n\in\N}$ in $\Pb\times\X$, with $(u_n,v_n)\to (u,v)$ as $n\to\infty$,
and $(t_n)_{n\in\N}$ in $(0,+\infty)$, with $t_n\downarrow 0$ as $n\to\infty$, such that
$$
  (\bar p,\bar x)+t_n(u_n,v_n)=(\bar p+t_nu_n,\bar x+t_nv_n)\in\graph\Solv,
  \quad\forall n\in\N,
$$
or, equivalently,
$$
  \bar x+t_nv_n\in\Solv(\bar p+t_nu_n)=C(\bar p+t_nu_n)\cap
  A(\bar p+t_nu_n,\bar x+t_nv_n)^{-1}(Q(\bar p+t_nu_n)),
  \quad\forall n\in\N.
$$
In particular, this inclusion implies that $\bar x+t_nv_n\in C(\bar p+t_nu_n)$,
for every $n\in\N$. Since $v_n\to v$ as $n\to\infty$, this shows that
\begin{equation} \label{in:vinDC1}
   v\in\Gder{C}{\bar p,\bar x}(u).
\end{equation}

Moreover, the same inclusion entails that $A(\bar p+t_nu_n,\bar x+t_nv_n)\in
Q(\bar p+t_nu_n)$, for every $n\in\N$. By the differentiability of $A$ at $(\bar p,\bar x)$,
it follows
$$
   A(\bar p,\bar x)+t_n\nabla A(\bar p,\bar x)(u_n,v_n)+o(\|t_n(u_n,v_n)\|)
   \in Q(\bar p+t_nu_n),\quad\forall n\in\N,
$$
where $o(\|w\|)$ means
$$
   \lim_{w\to\nullv}{\|o(w)\|\over \|w\|}=0.
$$
Therefore, by setting $y_n=A(\bar p,\bar x)(u_n,v_n)+t_n^{-1}o(\|t_n(u_n,v_n)\|)$,
one has
$$
  A(\bar p,\bar x)+t_ny_n\in  Q(\bar p+t_nu_n),\quad\forall n\in\N,
$$
with $y_n\to \nabla A(\bar p,\bar x)(u,v)$, as $n\to\infty$.
From this fact it possible to deduce the inclusion
$(u,\nabla A(\bar p,\bar x)(u,v))\in\Tang{\graph Q}{(\bar p,A(\bar p,\bar x))}=\graph
\Gder{Q}{\bar p,A(\bar p,\bar x)}$, and consequently $\nabla A(\bar p,\bar x)(u,v)\in
\Gder{Q}{\bar p,A(\bar p,\bar x)}(u)$. This means
\begin{equation} \label{in:vinA-1DQ2}
   v\in \nabla A(\bar p,\bar x)(u,\cdot)^{-1}\left(\Gder{Q}{\bar p,A(\bar p,\bar x)}(u)\right).
\end{equation}
The validity of both the inclusions (\ref{in:vinDC1}) and (\ref{in:vinA-1DQ2})
allows one to obtain the inclusion in the thesis.
\smartqed\qed
\end{proof}

\begin{remark}
The reader should notice that, whereas the positive homogeneity
of the map $\nabla A(\bar p,\bar x)$ is actually exploited, the additivity property
plays no role in the proof of Theorem \ref{thm:outapproxDSolv}. Therefore,
the resulting outer approximation can be readily generalized by replacing the
Fr\'echet differentiability assumption on $A$ with proper differentiability
notions from nonsmooth analysis. For example, following \cite{Robi91},
a map $F:\X\longrightarrow\Y$ between normed vector spaces is said to
be $B$-differentiable at $x_0$ if there is a continuous positively
homogeneous map $\Bder{F}{x_0}:\X\longrightarrow\Y$ such that
$$
  F(x_0+u)=F(x_0)+\Bder{F}{x_0}(u)+o(\|u\|).
$$
The function $\Bder{F}{x_0}$ is necessarily unique, whenever it exists, and
$\Bder{F}{x_0}=\nabla F(x_0)$, provided that $F$ is Fr\'echet differentiable at $x_0$.
If $A$ is $B$-differentiable at $(\bar p,\bar x)$, by trivial adjustments
in the proof of Theorem \ref{thm:outapproxDSolv}  it is possible
to prove the validity of the following inclusion
$$
  \Gder{\Solv}{\bar p,\bar x}(u)\subseteq \Gder{C}{\bar p,\bar x}(u) \cap
  \bigl[\Bder{F}{\bar p,\bar x}(u,\cdot)^{-1}\left(\Gder{Q}{\bar p,A(\bar p,\bar x)}(u)\right)\bigl],
  \quad\forall u\in\Pb,
$$
generalizing (\ref{GderSolvout}).
\end{remark}

The next example aims at illustrating that the inclusion in (\ref{GderSolvout})
may happen to be strict.

\begin{example}({\it Exactness failure})
Let $\Pb=\X=\Y=\R$ be endowed with their standard Euclidean space structure.
Consider the family $(\SFP_p)$ defined by
$$
  C(p)=(-\infty,-p^2],\qquad Q(p)=[p^2,+\infty), \qquad A(p,x)=px,
$$
and fix $\bar p=\bar x=0$, so $A(0,0)=0$. As
\begin{eqnarray*}
  px\ge p^2 \qquad\hbox{ iff } \qquad
  \left\{\begin{array}{ll}
  x\ge p, & \qquad\hbox{\ if } p>0,
   \\
   x\in\R, & \qquad\hbox{\ if } p=0,
   \\
   x\le p, & \qquad\hbox{\ if } p<0,
  \end{array}
  \right.
\end{eqnarray*}
then the solution map $\Solv:\R\rightrightarrows\R$ associated with
the present $(\SFP_p)$ turns out to be
\begin{eqnarray*}
  \Solv(p)=\left\{\begin{array}{ll}
      \varnothing , & \qquad\hbox{\ if } p>0,
   \\
      (-\infty,\min\{-p^2,\, p \}], & \qquad\hbox{\ if } p\le 0. \\
  \end{array}
  \right.
\end{eqnarray*}
Thus, it is clear that $(0,0)\in\graph\Solv$ and $\Tang{\Solv}{(0,0)}=
\{(p,x)\in\R\times\R:\ x\le p\le 0\}$. Consequently, one finds
\begin{eqnarray*}
  \Gder{\Solv}{0,0}(u)=\left\{\begin{array}{ll}
      \varnothing , & \qquad\hbox{\ if } u>0,
   \\
      (-\infty,u], & \qquad\hbox{\ if } u\le 0. \\
  \end{array}
  \right.
\end{eqnarray*}
Since, as one readily sees,
$$
  \Tang{\graph C}{(0,0)}=\R\times (-\infty,0],\qquad
  \Tang{\graph Q}{(0,0)}=\R\times (-\infty,0],
$$
and
$$
  \nabla A(0,0)=\langle (0,0),(u,v)\rangle=0,\quad\forall(u,v)\in\R\times\R,
$$
one has
$$
  \Gder{C}{0,0}(u)=(-\infty,0],\qquad
  \Gder{Q}{0,A(0,0)}(u)=[0,+\infty),\quad\forall u\in\R,
$$
and hence
$$
  \nabla A(0,0)(u,\cdot)^{-1}\left(\Gder{Q}{0,A(0,0)}(u)\right)=
  \R,\quad\forall u\in\R.
$$
Therefore, in the present case, for every $u>0$ one obtains
\begin{eqnarray*}
 \Gder{\Solv}{0,0}(u) &=& \varnothing
   \subsetneqq (-\infty,0]    \\
 &= &\Gder{C}{0,0}(u)\cap \bigl[\nabla A(0,0)(u,\cdot)^{-1}\left(\Gder{Q}{0,A(0,0)}(u)\right)\bigl].
\end{eqnarray*}
\end{example}

Providing only a one-side estimate of ${\rm D}_{\rm T}\Solv$, Theorem \ref{thm:outapproxDSolv}
needs to be complemented with a result about approximations from inside.
This will be done here by an approach relying on error bounds, as seen in Section
\ref{Sect:2.4}, which requires a more articulated set of assumptions.

\begin{theorem}[Inner approximation of ${\rm D}_{\rm T}\Solv$]  \label{thm:inapproxDSolv}
With reference to a family $(\SFP_p)$, let $(\bar p,\bar x)\in\graph\Solv$.
Suppose that:
\begin{itemize}
  \item[(i)] $C:\Pb\rightrightarrows\X$ is l.s.c. at $(\bar p,\bar x)$;
  \item[(ii)] $Q:\Pb\rightrightarrows\Y$ is l.s.c. at $(\bar p,A(\bar p,\bar x))$;
  \item[(iii)] $A$ is Fr\'echet differentiable at $(\bar p,\bar x)$;
  \item[(iv)] $(\DRC)$ holds near $(\bar p,\bar x)$.
\end{itemize}
Then, the following inclusions hold for every $u\in\Pb$
\begin{equation}   \label{GderSolvia1}
  \Gder{\Solv}{\bar p,\bar x}(u)\supseteq \Gder{C}{\bar p,\bar x}(u) \cap
  \bigl[\nabla A(\bar p,\bar x)(u,\cdot)^{-1}\left( \Ider{Q}{\bar p,A(\bar p,\bar x)}(u)\right)\bigl];
\end{equation}
\begin{equation}   \label{GderSolvia2}
  \Gder{\Solv}{\bar p,\bar x}(u)\supseteq \Ider{C}{\bar p,\bar x}(u) \cap
  \bigl[\nabla A(\bar p,\bar x)(u,\cdot)^{-1}\left( \Gder{Q}{\bar p,A(\bar p,\bar x)}(u)\right)\bigl].
\end{equation}
\end{theorem}

\begin{proof}
Proof of inclusion (\ref{GderSolvia1}).
Fix any $u\in\Pb$ and take an arbitrary element
$$
  v\in\Gder{C}{\bar p,\bar x}(u) \cap \bigl[
  \nabla A(\bar p,\bar x)(u,\cdot)^{-1}\left( \Ider{Q}{\bar p,A(\bar p,\bar x)}(u)\right)\bigl].
$$
If the above intersection happens to be empty, then the inclusion in
(\ref{GderSolvia1}) is trivially true.
In the light of the characterization of the graphical contingent derivative
provided in Lemma \ref{lem:grphderdist}(i), it suffices to show that
\begin{equation}   \label{eq:liminfdistSolv0}
     \liminf_{u' \to u \atop t\downarrow 0} \frac{\dist{\bar x+tv}{\Solv(\bar p+tu')}}{t}=0.
\end{equation}
Under the current hypotheses it is possible to apply Theorem \ref{thm:Solverbo}.
In particular, notice that the continuity property at $\bar p$ required on $A(\cdot,\bar x)$
is ensured by the Fr\'echet differentiability of $A$ at $(\bar p,\bar x)$.
According Theorem \ref{thm:Solverbo}, there exist positive $\eta$ and $\zeta$ such that the following error bound
holds true
$$
  \dist{x}{\Solv(p)} \le {\dist{A(p,x)}{Q(p)}+\dist{x}{C(p)}\over \tau(\delta)},
   \ \forall (p,x)\in \cball{\bar p}{\zeta}\times \cball{\bar x}{\eta},
$$
where $\delta>0$ and $\tau(\delta)>0$ are as in $(\DRC)$.
This means that one can choose $\epsilon>0$ in such a way that
$$
  \bar p+tu'\in\cball{\bar p}{\zeta} \hbox{ and } \bar x+tv\in\cball{\bar x}{\eta},
  \quad\forall u'\in\cball{u}{\epsilon},\ \forall t\in (0,\epsilon],
$$
so to obtain
\begin{eqnarray}     \label{in:erboSolveps}
  \dist{\bar x+tv}{\Solv(\bar p+tu')} &\le &\tau(\delta)^{-1}
  \biggl[\dist{A(\bar p+tu',\bar x+tv)}{Q(\bar p+tu')}   \\
  & + &\dist{\bar x+tv}{C(\bar p+tu')} \biggl],
   \quad\forall (u',t)\in \cball{u}{\epsilon}\times (0,\epsilon]. \nonumber
\end{eqnarray}
Observe that by the hypothesis (iii) the following estimate holds
\begin{eqnarray*}
  & \dist{A(\bar p+tu',\bar x+tv)}{Q(\bar p+tu')} \\
   & =\dist{A(\bar p,\bar x)+t\nabla A(\bar p,\bar x)(u',v)+
  o(\|t(u',v)\|)}{Q(\bar p+tu')} \\
   & =\dist{A(\bar p,\bar x)+t\nabla A(\bar p,\bar x)(u',v)\mp t \nabla A(\bar p,\bar x)(u,v)+o(\|t(u',v)\|}{Q(\bar p+tu')}  \\
   &\le  t\|\nabla A(\bar p,\bar x)(u'-u,\nullv)\|  \\
    & +\dist{A(\bar p,\bar x)+t\nabla A(\bar p,\bar x)(u,v)}{Q(\bar p+tu')} + \|o(\|t(u',v)\|)\|, \\
   &\forall (u',t)\in \cball{u}{\epsilon}\times (0,\epsilon].
\end{eqnarray*}
Since $v\in \nabla A(\bar p,\bar x)(u,\cdot)^{-1}\left( \Ider{Q}{(\bar p,A(\bar p,\bar x))}(u)\right)$,
it is true that $\nabla A(\bar p,\bar x)(u,v)\in \Ider{Q}{(\bar p,A(\bar p,\bar x))}(u)$. Consequently,
according to what stated in  Lemma \ref{lem:grphderdist}(ii), it must be
$$
  \lim_{u' \to u \atop t\downarrow 0} \frac{\dist{A(\bar p,\bar x)+t\nabla A(\bar p,\bar x)(u,v)}{Q(\bar p+tu')}}{t}=0.
$$
From the above estimate, the annihilation of this limit leads to
\begin{eqnarray}     \label{in:limdistAQt}
  \lim_{u' \to u \atop t\downarrow 0} {\dist{A(\bar p+tu',\bar x+tv)}{Q(\bar p+tu')} \over t}
  &\le & \lim_{u' \to u \atop t\downarrow 0}
  \biggl[ \|\nabla A(\bar p,\bar x)\|_\Lin\|u'-u\| \nonumber \\
  & + & \displaystyle{\|o(\|t(u',v)\|)\|\over t}\\
  &+&  {\dist{A(\bar p,\bar x)+tA(u,v)}{Q(\bar p+tu')} \over t}  \biggl]=0. \nonumber
\end{eqnarray}
On the other hand, since $v\in\Gder{C}{\bar p,\bar x}(u)$, according to the characterization
in  Lemma \ref{lem:grphderdist}(i) it is true that
\begin{equation}  \label{eq:vinDCu}
   \liminf_{u' \to u \atop t\downarrow 0} \frac{\dist{\bar x+tv}{C(\bar p+tu')}}{t}=0.
\end{equation}
Then, by combining (\ref{in:erboSolveps}), (\ref{in:limdistAQt}), and (\ref{eq:vinDCu}), one obtains
\begin{eqnarray*}
  \liminf_{u' \to u \atop t\downarrow 0} \frac{\dist{\bar x+tv}{\Solv(\bar p+tu')}}{t} &\le &
  \frac{1}{\tau(\delta)}\biggl[\lim_{u' \to u \atop t\downarrow 0} {\dist{A(\bar p+tu',\bar x+tv)}{Q(\bar p+tu')} \over t}  \\
   &+& \liminf_{u' \to u \atop t\downarrow 0} \frac{\dist{\bar x+tv}{C(\bar p+tu')}}{t}\biggl]=0.
\end{eqnarray*}
This shows that the inequality in (\ref{eq:liminfdistSolv0}) is satisfied, thereby completing the proof
of the inclusion in (\ref{GderSolvia1}).

The proof of the inclusion in (\ref{GderSolvia2}) can be conducted by a similar technique.
In this case, properly adapting the argument described above leads to
\begin{eqnarray*}
  \liminf_{u' \to u \atop t\downarrow 0} \frac{\dist{\bar x+tv}{\Solv(\bar p+tu')}}{t} &\le &
  \frac{1}{\tau(\delta)}\biggl[ \liminf_{u' \to u \atop t\downarrow 0}{\dist{A(\bar p+tu',\bar x+tv)}{Q(\bar p+tu')} \over t}  \\
   &+& \lim_{u' \to u \atop t\downarrow 0}\frac{\dist{\bar x+tv}{C(\bar p+tu')}}{t}\biggl]=0.
\end{eqnarray*}
This completes the proof.
\smartqed\qed
\end{proof}

\begin{remark}
Recalling that $\Gder{\Solv}{\bar p,\bar x}(u)$ is always a closed set, the inclusions in
(\ref{GderSolvia1}) and $(\ref{GderSolvia2})$ can be immediately refined by writing
$$
  \Gder{\Solv}{\bar p,\bar x}(u)\supseteq \cl\biggl[\Gder{C}{\bar p,\bar x}(u) \cap
  \bigl[\nabla A(\bar p,\bar x)(u,\cdot)^{-1}\left( \Ider{Q}{\bar p,A(\bar p,\bar x)}(u)\right)\bigl]\biggl]
$$
and
$$
 \Gder{\Solv}{\bar p,\bar x}(u)\supseteq \cl\biggl[\Ider{C}{\bar p,\bar x}(u) \cap
  \bigl[\nabla A(\bar p,\bar x)(u,\cdot)^{-1}\left( \Gder{Q}{\bar p,A(\bar p,\bar x)}(u)\right)\bigl]
  \biggl].
$$
\end{remark}

As a comment to Theorem \ref{thm:inapproxDSolv}, it must be remarked that
for all those $u\in\Pb$ for which $\Ider{Q}{\bar p,A(\bar p,\bar x)}(u)=\varnothing$
(resp. $\Ider{C}{\bar p,\bar x}(u)=\varnothing$) the inclusion in $(\ref{GderSolvia1})$
(resp. in $(\ref{GderSolvia2})$) becomes uninformative. Nevertheless, conditions
are known, which ensure the nonemptiness of the above sets. For the latter set,
this occurs for example if $\graph C$ is convex with nonempty interior.

Within variational analysis, derivatives of multifunctions have been employed
in the formulation of criteria for detecting their Lipschitzian behaviour.
One of the most investigated properties of this kind is the Aubin property
(see \cite{AubFra90,DonRoc14,Ioff17,Mord06}).
Recall that a set-valued map $F:\X\rightrightarrows\Y$ between Banach spaces
is said to have the Aubin property at $(x_0,y_0)\in\graph F$ if there exist
$\ell\ge 0$ and positive $\delta$ and $\eta$ such that
\begin{equation}    \label{in:defAubin}
   F(x_1)\cap\cball{y_0}{\eta}\subseteq\cball{F(x_2)}{\ell
   \|x_1-x_2\|},\quad\forall x_1,\, x_2\in\cball{x_0}{\delta};
\end{equation}
the value
$$
  \Lip{F}{x_0}{y_0}=\inf\{\ell\ge0:\ \exists\delta,\,\eta>0
  \hbox{ for which (\ref{in:defAubin}) holds}\}
$$
is called the exact Aubin property bound of $F$ at $(x_0,y_0)$. Notice that
this property actually takes place iff $\Lip{F}{x_0}{y_0}<+\infty$.
As an application of Theorem \ref{thm:inapproxDSolv}, it is possible to
establish a sufficient condition for $\Solv$ to satisfy the Aubin property
at $(\bar p,\bar x)\in\graph\Solv$, which, bypassing the calculation of
${\rm D}_{\rm T}\Solv$, relies on the problem data. To do so, given a set-valued
map $F:\X\rightrightarrows\Y$ between Banach spaces, let us introduce the
value
$$
  \|F\|_-=\sup_{u\in\Uball}\dist{\nullv}{F(u)}.
$$
In order to formulate the next result, given a family  $(\SFP_p)$ and
$(\bar p,\bar x)\in\graph\Solv$, set
$$
  \ell_{{}_{C,A,Q}}(\bar p,\bar x) = \limsup_{(p,x)\to (\bar p,\bar x)\atop (p,x)\in\graph\Solv}
   \|\Gder{C}{p,x}(u) \cap
  \bigl[\nabla A(p,x)(u,\cdot)^{-1}\left( \Ider{Q}{p,A(p,x)}(u)\right)\bigl]\|_-,
$$
where $\limsup_{(p,x)\to (\bar p,\bar x)\atop (p,x)\in\graph\Solv}$ meas that
$(p,x)\to (\bar p,\bar x)$ while satisfying the inclusion $(p,x)\in\graph\Solv$.

\begin{corollary}[Sufficient derivative condition for Aubin property]
With reference to a family $(\SFP_p)$, let $(\bar p,\bar x)\in\graph\Solv$.
Suppose that:
\begin{itemize}
  \item[(i)] $C:\Pb\rightrightarrows\X$ is l.s.c. in a neighbourhood of $(\bar p,\bar x)$;
  \item[(ii)] $Q:\Pb\rightrightarrows\Y$ is l.s.c. in a neighbourhood of $(\bar p,A(\bar p,\bar x))$;
  \item[(iii)] $A$ is Fr\'echet differentiable in a neighbourhood of $(\bar p,\bar x)$;
  \item[(iv)] $(\DRC)$ holds near $(\bar p,\bar x)$;
  \item[(v)] $\Solv$ is locally closed.
\end{itemize}
Then, it holds
\begin{equation}  \label{in:elllip}
   \Lip{\Solv}{\bar p}{\bar x}\le \ell_{{}_{C,A,Q}}(\bar p,\bar x).
\end{equation}
Therefore, if $\ell_{{}_{C,A,Q}}(\bar p,\bar x)<+\infty$ then $\Solv$ satisfies the Aubin property
at $(\bar p,\bar x)$.
\end{corollary}

\begin{proof}
By applying the well-known primal criterion for the Aubin property of
locally closed multifunctions between Banach spaces (see \cite[Corollary 5.23]{Ioff17}),
one has
$$
  \Lip{\Solv}{\bar p}{\bar x}\le \limsup_{(p,x)\to (\bar p,\bar x)\atop (p,x)\in\graph\Solv}
  \|\Gder{\Solv}{p,x}\|_-.
$$
Since it is possible to employ the inner approximation of ${\rm D}_{\rm T}\Solv$ given
by $(\ref{GderSolvia1})$ at each point of $\graph\Solv$ in the neighbourhood of $(\bar p,\bar x)$,
one obtains
$$
   \|\Gder{\Solv}{p,x}\|_-\le \|\Gder{C}{p,x}(u) \cap
  \bigl[\nabla A(p,x)(u,\cdot)^{-1}\left( \Ider{Q}{p,A(p,x)}(u)\right)\bigl]\|_-.
$$
Then it suffices to pass to the $\limsup$ as $(p,x)\to (\bar p,\bar x)$, with $(p,x)\in\graph\Solv$.
\smartqed\qed
\end{proof}

It is worth noticing that the estimate in (\ref{in:elllip}) can be achieved only by
means of the inner approximation of ${\rm D}_{\rm T}\Solv$ in $(\ref{GderSolvia1})$
(a similar estimate can be achieved by employing $(\ref{GderSolvia2})$),
whereas the outer approximation given by Theorem \ref{thm:outapproxDSolv} can not
play any role here.

\vskip.5cm

%%%%%%%%%%%%%%%%%%%%%%%%%%%%%%%%%%%%%%%%%%%%%%%%%%%%%%%%%%%%%%%%%%%%%%%%%%%%%%%%%%%%%%%%%%%%%%%%%%

%\section{First-order approximation in dual spaces}     \label{Sect:}

%\vskip.5cm

%%%%%%%%%%%%%%%%%%%%%%%%%%%%%%%%%%%%%%%%%%%%%%%%%%%%%%%%%%%%%%%%%%%%%%%%%%%%%%%%%%%%%%%%%%%%%%%%%%%%%%%

\section{An application to mathematical programming
with ``split" feasibility constraints}     \label{Sect:4}

In the present section, constrained optimization problems of the following
form will be considered
$$
  \min \vartheta(p,x)
  \qquad\hbox{ subject to }\qquad x\in\Solv(p),  \leqno (\MPSFC)
$$
where $\Solv:\Pb\rightrightarrows\X$ is the solution map associated with
a family of parametric ``split" feasibility problems $(\SFP_p)$,
defined as in (\ref{eq:Solvpsetchar}), whereas $\vartheta:\Pb\times\X
\longrightarrow\R$ represents the objective function.
This kind of problems essentially differ from standard problems
of mathematical programming with traditional equality and inequality
constraints. Indeed, the objective function depends on two variables,
playing different roles. There is one kind of variable (the so-called
``state" variable, here denoted by $x$), whose feasibility is conditioned
by the other variable (sometimes called ``control" variable, here denoted
by $p$). More precisely, the exact determination of the feasible region of $(\MPSFC)$
requires the resolution of a lower level family of possibly involved
problems, such as $(\SFP_p)$ may be. Besides, such kind of constraints
is known to be a source of nonsmoothness when dealing with its functional characterization.
Reasonably, the study of optimality conditions for $(\MPSFC)$ benefits
from insights into the local behaviour of $\Solv$.

Under a locally Lipschitz assumption on the objective function $\vartheta$,
an approach to deriving necessary optimality conditions for $(\MPSFC)$
can be developed by a standard penalization technique.

\begin{proposition}    \label{pro:noptcond0}
Let $(\bar p,\bar x)\in\Pb\times\X$ be a local solution to problem $(\MPSFC)$.
Suppose that:
\begin{itemize}
  \item[(i)]  $C:\Pb\rightrightarrows\X$ is l.s.c. at $(\bar p,\bar x)$;

  \item[(ii)] $Q:\Pb\rightrightarrows\Y$ is l.s.c. at $\left(\bar p,A(\bar p,\bar x)\right)$;

  \item[(iii)] the qualification condition $(\DRC)$ holds near $(\bar p,\bar x)$,
  with $\tau(\delta)>0$;

  \item[(iv)] $\vartheta$ is locally Lipschitz around $(\bar p,\bar x)$, with
  constant $l_\vartheta>0$.
\end{itemize}
Then, for every $\lambda\ge \tau(\delta)^{-1}l_\vartheta$, $(\bar p,\bar x)$
is an unconstrained local minimizer of the function $\vartheta_\lambda:\Pb\times\X
\longrightarrow\R\cup\{\pm\infty\}$, given by
$$
  \vartheta_\lambda(p,x)=\vartheta(p,x)+\lambda\mf(p,x).
$$
\end{proposition}

\begin{proof}
According to the well-known principle of exact penalization valid
for locally Lipschitz functions (see, for instance, \cite[Proposition 2.4.3]{Clark83}),
by the hypothesis (iv), for a proper $r>0$, $(\bar p,\bar x)$ turns out to be a local solution
to the unconstrained problem
$$
  \min \left[\vartheta(p,x)+l\dist{(p,x)}{\graph\Solv\cap\cball{(\bar p,\bar x)}{r}}\right]
$$
for every $l\ge l_\vartheta$. Recall that by a general property of the
distance function valid in any metric space setting, one has
\begin{equation}   \label{eq:distdistball}
  \dist{(p,x)}{\graph\Solv\cap\cball{(\bar p,\bar x)}{r}}=\dist{(p,x)}{\graph\Solv},
  \quad\forall (p,x)\in\cball{(\bar p,\bar x)}{r/2}.
\end{equation}
On the other hand, it holds
\begin{eqnarray}   \label{in:distgraphdist}
  \dist{(p,x)}{\graph\Solv} &=& \inf_{(q,z)\in\graph\Solv}\|(p,x)-(q,z)\| \nonumber \\
  &=& \inf\{\|p-q\|+\|x-z\|:\ z\in\Solv(q)\} \nonumber \\
  &\le & \inf_{z\in\Solv(p)}\|x-z\|=\dist{x}{\Solv(p)},\quad\forall (p,x)\in\Pb\times\X.
\end{eqnarray}
By combining the equality in (\ref{eq:distdistball}) with the inequality
in (\ref{in:distgraphdist}), in the light of the error bound estimate in
(\ref{in:serbo}) one obtains
\begin{eqnarray*}
  \vartheta(\bar p,\bar x) &\le & \vartheta(p,x)+l\dist{(p,x)}{\graph\Solv} \\
   &\le &  \vartheta(p,x)+l\dist{(x)}{\Solv(p)}  \\
   &\le & \vartheta(p,x)+\tau(\delta)^{-1}l\mf(p,x),\quad\forall
   (p,x)\in \cball{(\bar p,\bar x)}{r/2}.
\end{eqnarray*}
This completes the proof.
\smartqed\qed
\end{proof}

Proposition \ref{pro:noptcond0} can be taken as a starting point for the formulation
of more elaborated necessary optimality conditions, which can be established with
the aid of adequate tools on nonsmooth analysis.

Below some necessary optimality conditions are established,
which are actually based on some of the results exposed in
Section \ref{Sect:3} and do not rely on the penalization technique
exploited above. In doing so, let us point out that they do not require any locally Lipschitz
assumption on $\vartheta$.

\begin{proposition}[Graphical optimality condition]    \label{pro:noptcond1}
Let $(\bar p,\bar x)\in\Pb\times\X$ be a local solution to problem $(\MPSFC)$.
Suppose that:
\begin{itemize}
  \item[(i)] $C$ and $Q$ are closed convex multifunctions;

  \item[(ii)] $A\in\Lin(\X,\Y)$ is onto;

  \item[(iii)] $\nullv\in\inte[\graph C-\graph(A^{-1}\circ Q)]$.
\end{itemize}
Then, it must be
\begin{equation}  \label{in:optcon1}
    -\Upsubd\vartheta(\bar p,\bar x)\subseteq
    \cl \biggl[\dcone{\graph\Gder{C}{\bar p,\bar x}}+
    A^*\left(\dcone{\graph\Gder{Q}{\bar p,A(\bar x)}}\right)\biggl].
\end{equation}
\end{proposition}

\begin{proof}
As $(\bar p,\bar x)\in\graph\Solv$ is a local solution to $(\MPSFC)$,
there exists $r>0$ such that
\begin{equation}    \label{in:locsolMPSFC}
   \vartheta(\bar p,\bar x)\le\vartheta(p,x),\quad\forall
   (p,x)\in\graph\Solv\cap\cball{(\bar p,\bar x)}{r}.
\end{equation}
Take an arbitrary $w^*\in\Upsubd\vartheta(\bar p,\bar x)$.
According to the variational representation of $\Upsubd\vartheta(\bar p,\bar x)$
recalled in Section \ref{Sect:2.2}, there is $\varsigma:\Pb\times\X\longrightarrow\R$,
with $\varsigma(\bar p,\bar x)=\vartheta(\bar p,\bar x)$ and $\varsigma(p,x)\ge\vartheta(p,x)$
for every $(p,x)\in\Pb\times\X$, which is Fr\'echet differentiable at $(\bar p,\bar x)$,
with $\nabla\varsigma(\bar p,\bar x)=w^*$.
From the inequality in (\ref{in:locsolMPSFC}), it follows
\begin{equation}    \label{in:locsolsigma}
   \varsigma(\bar p,\bar x)=\vartheta(\bar p,\bar x)\le
   \vartheta(p,x)\le\varsigma(p,x),\quad\forall
   (p,x)\in\graph\Solv\cap\cball{(\bar p,\bar x)}{r}.
\end{equation}
Now, if $(u,v)\in\Tang{\graph\Solv\cap\cball{(\bar p,\bar x)}{r}}{(\bar p,\bar x)}=
\Tang{\graph\Solv}{(\bar p,\bar x)}$, then there must exist $(u_n,v_n)_{n\in\N}$ in $\Pb\times\X$,
with $(u_n,v_n)\to(u,v)$ as $n\to\infty$, and $(t_n)_{n\in\N}$ in $(0,+\infty)$, with
$t_n\downarrow 0$ as $n\to\infty$, such that
$$
  (\bar p,\bar x)+t_n(u_n,v_n)\in\graph\Solv\cap\cball{(\bar p,\bar x)}{r},
  \quad\forall n\in\N.
$$
On account of the inequality in (\ref{in:locsolsigma}), by the Fr\'echet differentiability
of $\varsigma$ at $(\bar p,\bar x)$, the last inclusion leads to
$$
  0\le {\varsigma(\bar p+t_nu_n,\bar x+t_nv_n)-\varsigma(\bar p,\bar x)\over
  t_n} =\langle w^*,(u_n,v_n)\rangle + {o(\|t_n(u_n,v_n)\|)\over t_n},
  \quad\forall n\in\N,
$$
whence, by passing to the limit as $n\to\infty$, one obtains
\begin{equation}    \label{in:w*GderSolv}
   \langle w^*,(u,v)\rangle\ge 0,\quad \forall (u,v)\in\graph\Gder{\Solv}{\bar p,\bar x}.
\end{equation}
In the light of the exact representation of $\Gder{\Solv}{\bar p,\bar x}$ provided
by Theorem \ref{thm:convspecase}, which is valid
under the current assumptions, the inequality in (\ref{in:w*GderSolv}) implies
$$
  -w^*\in\dcone{\graph\Gder{\Solv}{\bar p,\bar x}}=
  \dcone{\bigl[\graph\Gder{C}{\bar p,\bar x}\cap A^{-1}(\graph\Gder{Q}{\bar p,A(\bar x)})\bigl]}.
$$
By virtue of the calculus rule for the dual cone recalled in (\ref{eq:dconecalrule}), one obtains
$$
   -w^*\in\cl \biggl[\dcone{\graph\Gder{C}{\bar p,\bar x}}+
    A^*\left(\dcone{\graph\Gder{Q}{\bar p,A(\bar x)}}\right)\biggl].
$$
By arbitrariness of $w^*\in\Upsubd\vartheta(\bar p,\bar x)$, this
completes the proof of the inclusion in the thesis.
\smartqed\qed
\end{proof}

To assess the impact of the optimality condition formulated in Theorem \ref{pro:noptcond1},
notice that no useful information is carried whenever $\Upsubd\vartheta(\bar p,\bar x)=\varnothing$.
This happens, for example, if $\vartheta$ is a convex continuous function,
which fails to be Fr\'echet differentiable at $(\bar p,\bar x)$. Helpfully, the upper subdifferential
is nonempty for large classes of functions, including the Fr\'echet differentiable functions,
concave continuous functions, and the semiconcave functions.
In all such cases, the necessary optimality condition provided by $(\ref{in:optcon1})$
may be more efficient than those expressed in terms of more traditional lower subdifferentials.
This because it requires
that all elements in $ -\Upsubd\vartheta(\bar p,\bar x)$ belong to the set in the right-side
of $(\ref{in:optcon1})$, in contrast to a mere nonempty intersection requirement,
which is typical for the lower subdifferential case (see on this
aspect the discussion in \cite[Section 5.1]{Mord06b}).

Under additional assumptions on the constraint system, the current approach allows one to
establish the below necessary optimality condition, which separates
the role of the two components of elements in $\Upsubd\vartheta(\bar p,\bar x)$.

\begin{proposition}   \label{pro:noptcond2}
Let $(\bar p,\bar x)\in\Pb\times\X$ be a local solution to problem $(\MPSFC)$.
Suppose that all the hypotheses of Proposition \ref{pro:noptcond1} are satisfied and
suppose, in addition, that one of the following conditions is satisfied for every $u\in\Pb$:
\begin{itemize}
  \item[$(iv_1)$] $\inte\left(\Gder{C}{\bar p,\bar x}(u)\right)\cap\left(A^{-1}(\Gder{Q}{\bar p,A(\bar x)}(u))\right)
  \ne\varnothing$;

  \item[$(iv_2)$] $\nullv\in\inte\bigl[\Gder{C}{\bar p,\bar x}(u)-\left(A^{-1}(\Gder{Q}{\bar p,A(\bar x)}(u))\right)\bigl]$
  and at least one among the sets $\Gder{C}{\bar p,\bar x}(u)$ and $A^{-1}(\Gder{Q}{\bar p,A(\bar x)}(u))$ is bounded.
\end{itemize}
Then, for every $w^*=(u^*,v^*)\in \Upsubd\vartheta(\bar p,\bar x)$ it must be
\begin{equation}  \label{in:optcon2}
    -u^*\in\partial\biggl(-[\supp{v^*}{-\Gder{C}{\bar p,\bar x}}\Box
    \supp{v^*}{-\Gder{Q}{\bar p,A(\bar x)}}]\biggl)(\nullv).
\end{equation}
\end{proposition}

\begin{proof}
Fix an arbitrary $w^*=(u^*,v^*)\in \Upsubd\vartheta(\bar p,\bar x)$.
Under the above assumptions, the inequality in (\ref{in:w*GderSolv}) says
that it must be
$$
   \langle w^*,(u,v)\rangle=\langle u^*,u\rangle+\langle v^*,v\rangle\ge 0,
   \quad \forall v\in\Gder{\Solv}{\bar p,\bar x},\
   \forall u\in\Pb.
$$
By recalling the exact representation of $\Gder{\Solv}{\bar p,\bar x}$
provided in (\ref{GderSolvexact}), from the above inequality one obtains
\begin{eqnarray}    \label{in:u*useparv*v}
  \langle u^*,u\rangle &\ge & \sup_{v\in -\Gder{\Solv}{\bar p,\bar x}(u)}\langle v^*,v\rangle
  =\supp{v^*}{ -\Gder{\Solv}{\bar p,\bar x}(u)}     \nonumber \\
   &=& \supp{v^*}{-\Gder{C}{\bar p,\bar x}(u) \cap  -A^{-1}\left( \Gder{Q}{\bar p,A(\bar x)\right)}(u)},
   \quad\forall u\in\Pb.
\end{eqnarray}
Observe that both the set $\Gder{C}{\bar p,\bar x}(u)$ and $-A^{-1}\left( \Gder{Q}{\bar p,A(\bar x)\right)}(u)$
are closed and convex. Thus, under one of the hypotheses $(iv_1)$ or $(iv_2)$, the exact behaviour
of support functions recalled in (\ref{eq:suppinter}) allows one to write
\begin{eqnarray*}
  \supp{v^*}{-\Gder{C}{\bar p,\bar x}(u) \cap  -A^{-1}\left( \Gder{Q}{\bar p,A(\bar x)\right)}(u)}  =   \\
  \supp{v^*}{-\Gder{C}{\bar p,\bar x}(u)}\Box \supp{v^*}{ -A^{-1}\left( \Gder{Q}{\bar p,A(\bar x)\right)}(u)},
  \quad\forall u\in\Pb.
\end{eqnarray*}
Now, notice that, since $v^*$ remains fixed, both the functions
$$
  u\mapsto\supp{v^*}{-\Gder{C}{\bar p,\bar x}(u)} \quad\hbox{ and }\quad
  u\mapsto\supp{v^*}{ -A^{-1}\left( \Gder{Q}{\bar p,A(\bar x)\right)}(u)}
$$
are concave and positively homogeneous on $\Pb$, because under the current assumptions
the set-valued maps $\Gder{C}{\bar p,\bar x}:\Pb\rightrightarrows\X$ and
$A^{-1}\circ \Gder{Q}{\bar p,A(\bar x)}:\Pb\rightrightarrows\X$ are closed, convex
and positively homogeneous (remember Remark \ref{rem:concsupF}).
From the inequality in (\ref{in:u*useparv*v}),
it is possible to deduce that the sublinear function
$$
  u\mapsto \langle u^*,u\rangle -[\supp{v^*}{-\Gder{C}{\bar p,\bar x}(u)}\Box
  \supp{v^*}{ -A^{-1}\left( \Gder{Q}{\bar p,A(\bar x)\right)}(u)}]
$$
admits a (global) minimizer at $\nullv$. Consequently, according
to a well-known optimality condition from convex analysis, it must be
$$
 \nullv^*\in\partial\biggl(u^*-[\supp{v^*}{-\Gder{C}{\bar p,\bar x}}\Box
    \supp{v^*}{-\Gder{Q}{\bar p,A(\bar x)}}]\biggl)(\nullv).
$$
Then, the inclusion in the thesis becomes an obvious consequence
of the Moreau-Rockafellar sum rule for the subdifferential of convex
analysis.
\smartqed\qed
\end{proof}

\vskip.5cm

%%%%%%%%%%%%%%%%%%%%%%%%%%%%%%%%%%%%%%%%%%%%%%%%%%%%%%%%%%%%%%%%%%%%%%%%%%%%%%%%%%%%%%%%%%%%%%%%%%%%%%%

\section{Conclusions}

The results established in the paper provide formulae for
estimating the first-order contingent derivative of the solution set-valued map associated
with a family of parameterized ``split" feasibility problems.
A case is singled out, in which these formulae are able to give an exact representation
of the contingent derivative, whereas in general they result in one-side approximations.
Such kind of results appear as a
quantitative statement in the formulation of implicit multifunction theorems, as conceived
in modern variational analysis, and complement the perturbation
analysis of ``split" feasibility problems, as started in \cite{HuXuYe25,HuXuYe25b}
and continued in \cite{Uder26}, where all problem data are considered to be subject to perturbation.
Similarly as in classical implicit function theorems, the achieved
formulae are built on (various kinds of) derivatives of the problem data.
Some of them work under qualification conditions, which are expressed
in terms of convex analysis constructions.

Since the early development of the Euler-Lagrange theory of optimization
problems, the first-order approximation of solution sets to constraint
systems was well understood to play a crucial role in deriving optimality conditions.
The results exposed in the applicative section of the paper provide
some proposals for addressing such a connection in the specific context
of mathematical programming with `split" feasibility constraints.

\vskip.5cm

%%%%%%%%%%%%%%%%%%%%%%%%%%%%%%%%%%%%%%%%%%%%%%%%%%%%%%%%%%%%%%%%%%%%%%%%%%%%%%%%%%%%%%%%%%%%%%%%%%%%%%%

%\begin{acknowledgements}
%\end{acknowledgements}

%%%%%%%%%%%%%%%%%%%%%%%%%%%%%%%%%%%%%%%%%%%%%%%%%%%%%%%%%%%%%%%%%%%%%%%%%%%%%%%%%%%%%%%%%%%%%%%%%%%%%%

%References

\end{document}